\documentclass[letterpaper]{article}
\usepackage{amsthm}

\usepackage{amsmath}
\usepackage{amsfonts, latexsym, amssymb} 
\usepackage{color}
\usepackage{graphicx}
\usepackage{placeins}

\usepackage{multicol}

\usepackage{rotating}
\usepackage{makecell}

\usepackage{algorithm}
\usepackage{algorithmic}
\newtheorem{assumption}{Assumption} 

\newcommand{\iid}{i.\@i.\@d.\@ }
\newcommand{\eg}{\emph{e.g.}, }
\newcommand{\ie}{\emph{i.e.}, }

\providecommand{\abs}[1]{\left\lvert#1\right\rvert}
\providecommand{\norm}[1]{\left\lVert#1\right\rVert}

\newcommand{\R}{{\mathbb R}}

\renewcommand{\P}{{\mathbb P}}

\newcommand{\neww}[1]{}

\renewcommand{\hat}{\widehat}
\renewcommand{\leq}{\leqslant}
\renewcommand{\geq}{\geqslant}

\newtheorem{theorem}{Theorem}
\newtheorem{proposition}[theorem]{Proposition}

\usepackage[parfill]{parskip} 

\usepackage{booktabs} 
\usepackage{array} 
\usepackage{paralist} 
\usepackage{verbatim} 

\usepackage{etex,etoolbox}
\usepackage{thmtools}
\usepackage{environ}

\begin{document}

 \title{$r$-Extreme Signalling for Congestion Control}

  \author{
 Jakub Mare{\v c}ek$^{a}$\thanks{Jakub is the corresponding author. Email: \{ jakub.marecek, robshort \}@ie.ibm.com, jiayuan.yu@concordia.ca},
     Robert Shorten$^{ab}$,
     Jia Yuan Yu$^c$\\
\footnotesize a: IBM Research, Technology Campus Damastown, Dublin 15,
     Ireland \\
\footnotesize b:  University College Dublin, SEEE, 
 Belfield, Dublin   4, Ireland \\
\footnotesize c:  Concordia University, 
1455 de Maisonneuve Blvd. West, Montr{\' e}al H3G 1M8, Canada
}

\maketitle

\begin{abstract}
In many ``smart city'' applications, congestion arises 
 in part due to the nature of signals received by individuals from a central authority.
In the model of Mare{\v c}ek et al.\ [Int. J. Control 88(10), 2015], 
 each agent uses one out of multiple resources at each time instant. 
The per-use cost of a resource depends on the number of concurrent users. 
A central authority has up-to-date knowledge of the congestion across all resources
 and uses randomisation to provide a scalar or an interval for each resource at each time.
In this paper, the interval to broadcast per resource is obtained by taking the minima and maxima of costs observed within a time window of length $r$,
rather than by randomisation.
We show that the resulting distribution of agents across resources also converges in distribution,
 under plausible assumptions about the evolution of the population over time.
\end{abstract}





\section{Introduction}

In many applications \cite{marevcek2015signaling}, 
 a number of agents need to use one out of a number of resources, 
 whose cost of use, per-agent, depends on the number of agents using the resource, concurrently. 
In addition to the agents, there is often also a central authority in charge of all the resources, with a complete and up-to-date information about their use.
The central authority may or may not provide information to the agents.
If no information is provided, the agents may choose the resources randomly,
  or using some simple policies \cite{gittins2011multi}.
If the central authority provides one scalar for each resource at each time instant, 
 all agents may compare the scalars across the resources available to them in the same way, 
 and make the same choice. 
Thereby, the usage of resources with the lowest announced scalar may increase sharply,
 while the usage of resources with the highest announced scalar may drop sharply,
 ultimately leading to a cyclic outcome.
As an alternative, Mare{\v c}ek et al.\ \cite{marevcek2015signaling} studied the use of randomisation in information provision, including the use of randomisation in deriving 
 an interval to be broadcast for each resource at each time instant. 
Many challenges remain, though. For one, the use of randomisation, or obfuscation,
  may be difficult to justify in practice.
The challenges 
 are intimately related to control, but little studied, so far.

Let us motivate our study by illustrating the cyclic outcome on 
 the example of roads during the rush hour.
Travel times are influenced by the number of people on the roads.
Congestion arises, when too many people want to use a particular road at the same time.
This is not necessarily due to the inherent capacity limits of
  the road, but often due to the ``synchronised'' manner of travel, and 
  the lack of foresight into the choices of other people.
Imagine that there are two roads of similar capacity from one section of a ring-road to the city center, 
 and a central authority announces the travel times on the radial roads as 10 and 20 minutes, respectively.
This may cause congestion on the first radial road, in the short term, and lead to the
 congestion alternating between the two radial roads, subsequently.
In Appendix~\ref{sec:app1}, we show that under simplistic assumptions,
a similar limit-cycle behaviour could be observed 
for \emph{any} approach that picks a scalar to broadcast for each resource, 
as long as the scalars are distinct across resources,
and that it can lead to an arbitrarily bad behavior.
In practice, the differences due to signalling are bounded, 
but the example suggests why we aim to reduce the synchronisation.

In this paper, we hence study the problem of information provision, which is:
\begin{itemize}
\item non-stationary, inasmuch the costs associated with resources are not stationary, but rather influenced by the agents' actions
\item populational, inasmuch the agents come in a variety of types,
with a population described by a distribution over the types
\item limited in terms of feedback, inasmuch the agent has access only to aggregate information about the state of each resource, provided by the central authority
\item limited in terms of the agents' memory, 
inasmuch the agents pick the resource based on the most recently provided piece of information for each resource.
\end{itemize}
The paper is structured as follows.
Section~\ref{sec:mod} formalises the problem of information provision and suggests how a central authority can ``de-synchronise'' actions of people on the roads by providing them with signals. In particular, it suggests a signalling scheme,
 where one interval is broadcast for each route, with the additional constraint 
 that each interval remains consistent with past observations.  
Our main theorem in Section~\ref{sec:extreme} shows that if the
population comprises of agents risk-averse to varying degrees over time, 
we can improve the social outcome using interval
signaling with the intervals formed by extremes of the values encountered so far.  
In Section~\ref{sec:sim}, we demonstrate the considerable impact 
in simulations.
We conclude with an overview of related and potential future work in Sections
\ref{sec:related} and \ref{sec:conclusions}, respectively.

\section{Model}
\label{sec:mod}

We consider a dynamic discrete-time model of congestion,
suggested above and illustrated in Figure~\ref{fig:diag2}.
First, we describe the actions, then the signals, and finally 
the response of the population to the signals, which can be 
seen as a mapping from signals to actions.

\begin{figure}
\centering
\includegraphics[width=0.79\textwidth,clip=true,trim=8.5cm 7.5cm 6cm 4.5cm]{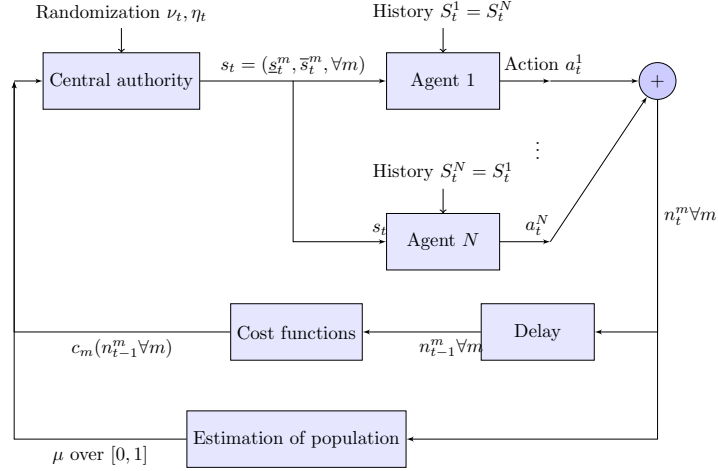}
  \caption{Block diagram of our model: Central authority
sends signal $s_t$ to each agent $1 \ldots N$.	 Each agent $i$ picks
an action $a_t^i$. Numbers $n^m_t$ of agents picking actions $A_m$ at time $t$,
summarise the state of the system at time $t$.
There is a social cost corresponding  to the state of the system, 
which the central authority uses to generate its signal.
}
  \label{fig:diag2}
\end{figure}

\subsection{Actions and their Costs}

A finite population of $N$ agents is confronted with $M$ alternative
choices at every time step.  The alternative actions are denoted
by $\{A_1,\ldots,A_M\}$ and time is discretised into periods
$t=1,2,\ldots$.  Let $a_t^i$ denote the choice of agent $i$ at time
$t$ and $n^m_t = \sum_i 1_{[a_t^i = A_m]}$ be the number of agents
choosing action $A_m$ at time $t$.  Throughout the paper, we assume
that each agent has to pick one of the $M$ actions at every time $t$.

The alternative actions $\{A_1,\ldots,A_M\}$ are perfectly
substitutable, \ie each agent decides only based on the cost.  The
cost of action $A_m$ at time $t$ is a function of the
number $n^m_t$ of agents that pick $A_m$ at time $t$.  We let $n_t$
denote the vector $(n^1_t , \ldots, n^M_t)$.  Let $c_m: \mathbb N \to
\mathbb R_+$ denote the so-called cost function for action $A_m$.  If
$n^m_t$ agents choose action $A_m$ at time $t$, the cost of action
$A_m$ at time $t$ to any single of them is $c_m(n^m_t)$.  We assume
that all $\{c_m\}$ are continuous.
Figure~\ref{fig:ushaped5settings} gives an example of two cost
functions.

\begin{figure}[t!]
  \includegraphics[clip=true,width=0.49
  \textwidth]{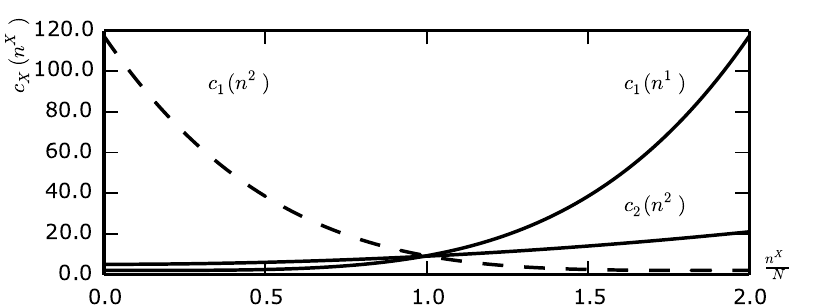}
  \includegraphics[clip=true,width=0.49
  \textwidth]{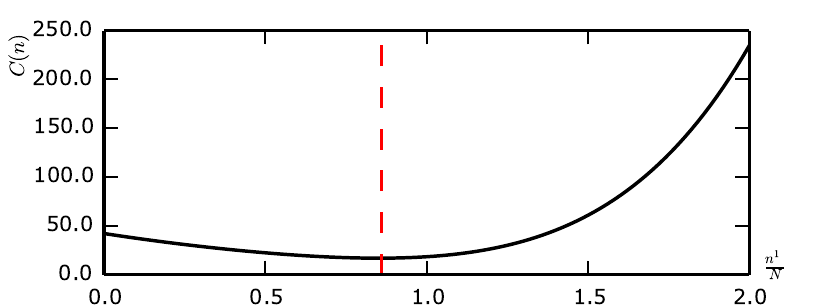}
                
  \caption{A trivial example with $N=2$. Left: Cost functions $c_1(x)
    \triangleq 2 (1 + 3.6 x^4)$ and $c_2(y) \triangleq 5 (1 + 0.8 y^2)$.
    Notice that $c_1$ is increasing in $x$ (solid line), but decreasing in $y = 2 - x$ (dashed).
    The structure of the cost functions, albeit seemingly arbitrary, 
    is customary in transportation engineering, as explained in Section~\ref{sec:sim}. 
    Right: The corresponding social cost.
    Notice that the minimum is at 0.86 (dashed vertical line), 
    which renders policies picking the action uniformly at random
    suboptimal.  
}
  \label{fig:ushaped5settings}
\end{figure}

The social cost $C(n_t)$ weights the costs of the two actions at
time $t$ with the proportions of agents taking the two actions, \ie
\begin{align}
  \label{eqn:socialCost}
  C(n_t) \triangleq \sum_{m=1}^M \frac{n^m_t}{N} \cdot c_m(n^m_t).
\end{align}
The social cost corresponding to the example cost functions is shown
in Figure~\ref{fig:ushaped5settings}.  Of further interest is the
time-averaged social cost:
\begin{align}
  \label{eqn:meanSocialCost}
  \hat C_T \triangleq \frac{1}{T} \sum_{t=1}^{T} C(n_t).
\end{align}

We study a number of signaling schemes and responses from agents.

\subsection{Signaling Schemes}

We introduce signaling schemes, which 
communicate information about the past 
cost of the $M$ actions.
Let $H_t$ denote the history of congestion costs up to time $t$:
\begin{align*}
  H_t \triangleq \{c_m(n^m_j) : m=1,\ldots,M, j=1,\ldots,t\}.
\end{align*}
Let $\mathcal H_t$ denote the set of all possible histories at time
$t$.  For a fixed integer $d$, a \emph{signaling scheme} is a set of
mappings $ \{s_t : \mathcal H_t \to \mathbb R^d \mid t=1,2,\ldots \}$,
where $s_t$ denotes the signal that the central authority broadcasts to
all agents at time $t$.

In \emph{scalar} signaling schemes, we have $d=M$,
one scalar value for each action.  
In \emph{interval} signaling schemes, $d = 2M$, and $s_t = (\underline
u_t^m, \overline u_t^m : m=1,\ldots, M)$, with $\underline u_t^m \le
\overline u_t^m$.  Notice that scalar signaling
schemes are equivalent to interval signaling schemes with $\underline
u_t^m = \overline u_t^m$, but may perform worse, by an arbitrary amount, 
as per Appendix~\ref{sec:app1}.
  Notice that these signaling schemes
summarise the history of observations $H_t$.




In a signaling scheme that
  we call \emph{$r$-extreme}, 
  for any fixed positive integer $r$,
  the central authority broadcasts
  the same signal $s_{t,r} = (\underline u_{t,r}^m, \overline
  u_{t,r}^m:m=1,\ldots,M)$ to all agents at time $t$, where
  \begin{align*}
    \underline u_{t,r}^m &= \min_{j=t-r,\ldots,t-1} \{ c_m(n^m_j) \},\\
    \overline u_{t,r}^m &= \max_{j=t-r,\ldots,t-1} \{ c_m(n^m_j) \}
    \quad\mbox{for all }m.
  \end{align*}

%

In a signaling scheme that we call \emph{$r$-subinterval}, 
for any fixed positive integer $r$,
the central authority broadcasts a signal $s_{t,r} = (\underline
w_{t,r}^m, \overline w_{t,r}^m: m=1,\ldots,M)$ to all agents at time $t$, such that
\begin{align}
  \max_{j=t-r,\ldots,t-1} \{ c_m(n^m_j) \} \geq \overline w_{t,r}^m \geq
  \underline w_{t,r}^m &\geq \min_{j=t-r,\ldots,t-1} \{ c_m(n^m_j) \},
  \label{eq:supported4}
\end{align}
for all $m$.
Notice that extreme signaling is a special case of subinterval signaling.

\subsection{Agent Population and Policies}

In response to the history of signals received prior to time $t$ and including it,
every agent $i$ takes action $a_t^i$.  For example, this action can be
a function of only the signal at a single time step $t$.  We
assume that
every agent acts based only on the signals,
without considering the response of other agents to its own action.
%
This is a reasonable assumption for three reasons.  First, it is hard
for the agent to obtain more information than the signal sent by the
central authority.  Second, the agents know that the signals received are
consistent with past observations.  Finally, when there is a large
number of agents, each has a very limited effect on the population as
a whole. 

Formally, let $S_t$ denote the history of signals broadcast up to time
$t$:
\begin{align*}
  S_t \triangleq \{s_1, \ldots, s_{t-1}\}.
\end{align*}
Let $\mathcal S_t$ denote the set of possible realisations of signal
histories 
up to time $t$.  A mapping of a signal history to an action, $\mathcal
S_t \to \{A_1,\ldots,A_M\}$, is called a policy.  We assume that the
number of agents $N$ is fixed over time.  We let $\Omega$ denote the
set of all possible types of agents.
Each type $\omega \in \Omega$ is associated with a policy, and every
agent of type $\omega$ follows the policy $\pi^\omega: \mathcal S_t
\to \{A_1,\ldots,A_M\}$.

We model the evolution of the number of agents of each type as
follows.  Let $\{\eta_k : \Omega \to \mathbb R \mid k=1,\ldots,K\}$
denote a finite set of probability measures over $\Omega$.  For
instance, for each subset $O \subseteq \Omega$, $\eta_k(O)$ can be
interpreted as a fraction of agents with policy $\pi^{\omega}$, except
that, for simplicity of analysis, the product $\eta_k(O) N$ does not
have to be an integer.  We let $(d_1,\ldots,d_K)$ denote a probability
measure over $(\eta_1,\ldots,\eta_K)$, \ie a probability measure over
a set of probability measures over $\Omega$.  The distribution of
agents among types $\Omega$ over time steps $t=1,2,\ldots$ is an \iid
sequence of random variables $\{\mu_t\}$, where the distribution of
$\mu_t$ is defined as $\P(\mu_t = \eta_k) \triangleq d_k$ for all $k$.
This allows us to model a population of agents that changes over time,
\eg one driver leaves the road network and is replaced by another
driver, with another policy.  For simplicity, we call $\mu_t$ the
\emph{population profile} at time $t$.

\subsection{$\pi^\omega$-policies}
\label{sec:policies}
In the case of $r$-extreme signaling, 
we consider a set of agent types $\Omega$, which is a finite set of
numbers, each of which is within $[0,1]$.
Recall that every agent $i$ receives the interval signal $s_t :=
(\underline{u}_t^m, \overline{u}_t^m:m=1,\ldots,M)$.  In response, we
assume that each agent $i$ of type $\omega$ follows the policy
$\pi^\omega$:
\begin{align}
  a_t^i = \pi^\omega(S_{t}) \in \left\{ \arg\min_{m \in \{1,\ldots,M\}} \omega
  \underline{u}^m_t + (1-\omega) \overline{u}^m_t \right\}, \label{eq:policy}
\end{align}
with the minimiser chosen uniformly at random, if non-unique.
This policy is a greedy heuristic, 
which seems natural, when one considers the following special
cases:
\begin{itemize}
\item \emph{Risk-seeking} $\omega = 1$, \ie acts based only on the
  best-case elements $(\underline{u}_t^m:m=1,\ldots,M)$
\item \emph{Risk-averse} $\omega = 0$, \ie acts based only on the
  worst-case elements $(\overline{u}_t^m:m=1,\ldots,M)$
\item \emph{Risk-neutral} $\omega = 0.5$, \ie acts based on the
  midpoints $\big((\underline{u}_t^m + \overline{u}_t^m)
  /2:m=1,\ldots,M \big)$.
\end{itemize}
Notice that this policy \eqref{eq:policy} could also model
convexifying ``multi-objective'' agents, \eg 90\% risk-seeking and
10\% risk-averse.

%
%
%
%
%


\section{A Stability Analysis} 
\label{sec:extreme}

In this section, we analyze the impact of $r$-extreme signaling.  We show that it
is stable in the sense that the population profile converges in
distribution under mild assumptions.  In particular, we study the case, where
the parameter $r$ is the function $r(t) = t$ of the time step $t$.

\begin{assumption}[\iid $\mu_t$, ``Population Renewal'']\label{as:mu1}
\mbox{} \\
  The distribution $\mu_t$ is
  an \iid sequence of random variables with $\P(\mu_t = \eta_k) = d_k$,
 $0 < d_k < 1$ for all $t$ and $k$, $\sum_{k} d_k = 1$.
\end{assumption}


Notice that the population renewal assumption does not entail the
elements of the random vector $\mu_t$ being independent.  We show that
under the population renewal assumption, the congestion profile
$(n^1_t, n^2_t)$ converges
for $\ell$-Lipschitz continuous cost functions for $\ell < 1/(Nk')$.
Observe that, for example, the function $c_m(x) = x/N$ is
 $(1/N)$-Lipschitz. 

\begin{theorem}[``Asymptotic Stability''] 
  \label{thm:conv}
  There exists a constant $k'$, such that under Assumption~\ref{as:mu1},
  if the functions $\{c_m:m=1,\ldots,M\}$ are $\ell$-Lipschitz
  continuous for $\ell < 1/(N k')$, there exists a unique limit, an $M$-dimensional random
  variable $Z$ such that the \emph{congestion profile} $(n_t^1/N, n_t^2/N, \ldots, n_t^M/N)$
  converges to $Z$ in distribution as $t\to\infty$.  
\end{theorem}

The proof relies on the following result from iterated function systems,
only trivially adapted from Barnsley et 
  al. {\cite{barnsley1988invariant,barnsley1989recurrent}}:

\begin{proposition}\label{thm:be}
  Let us have an index set $K$,
  a family $H$ of functions $\mathbb R^{2M} \to \mathbb R^{2M}$ indexed by $K$,
  and a family of real numbers $d$ indexed by $K$, $\sum_{k \in K} d_k = 1$. 
  Let us have another infinite family $w$, where for all $i_n$, 
  $w_{i_n} \mathbb R^{2M} \to \mathbb R^{2M}$ be \iid such that $\P(w_{i_n} = H_k) = d_k$ for all
  $k$.  
  If, for all $x, y \in \mathbb R^{2M}, x \not = y$,
  \begin{align*}
    \sum_{k \in K} d_k \log \left( \frac{\norm{H_k (x) - H_k (y)}_1}{\norm{x-y}_1)} \right) < 0,
  \end{align*}
  then the limit $\lim_{n} w_{i_n} ( \cdots w_{i_1} (x) \cdots )$ exists and is
  independent of $x$.
\end{proposition}







\begin{proof}
  \label{proof1}
  The proof proceeds in two steps.  First, we show that the signal
  process $s_t$ is an iterated function system. Then, we show that it
  converges in distribution.

  \textbf{(Step 1)}\\
  In order to apply Proposition~\ref{thm:be}, we construct an iterated
  function system in $\mathbb R^{2M}$. 
  First, let us recall the definitions introduced previously:
  \begin{align}
  s_{t+1} & \triangleq (\underline u^1_{t+1}, \overline u^1_{t+1}, 
                        \underline u^2_{t+1}, \overline u^2_{t+1}, \ldots,
                        \underline u^M_{t+1}, \overline u^M_{t+1})\label{defS}
  \end{align}
  where for all $m$:
  \begin{align}
    \underline u^m_{t+1} &\triangleq \min_j \{c_m(n^m_j) : j=1,\ldots,t \} \label{eq:28}   \\
    &= \min \{\underline u^m_t, c_m(n^m_t)\},\notag \\
    \overline u^m_{t+1} &= \max \{\overline u^m_t, c_m(n^m_t)\}. \notag 
  \end{align}
  Next, let us see that $n^m_t$ is a random variable:
  \begin{align}
    n^m_t &= \sum_i 1_{(a^i_t = A_m)}  \\
    &= \sum_i \sum_{\omega \in \Omega} 1_{(a^i_t = A_m \mid \textrm{ agent } i \textrm{ is of type }
      \omega
      )} \mu_t(\omega) \notag \\
    &= N \sum_{\omega \in \Omega} 1_{ \bigwedge_{s \neq r} ( \omega
      \underline u^m_t + (1-\omega) \overline u^m_t < \omega
      \underline u^s_t + (1-\omega) \overline u^s_t )}
    \mu_t(\omega)\label{eq:29}.
  \end{align}
  Observe that by Assumption~\ref{as:mu1}, the sequence $\{\mu_t\}$ is
  \iid.  Hence, $(n^m_t \mid 1 \le m \le M)$ is a random variable.

  Recall that $\Omega$ is finite and the support of the random variable $\mu_t$ is a finite
  set $\{\eta_1, \ldots, \eta_K\}$ of probability measures, where each
  \begin{align}
    \eta_k = (\eta_k(1) , \ldots, \eta_k(\abs{\Omega}) )
  \end{align}
  is such that
  \begin{align}
    \eta_k(j) &\in \{ i / N : i = 0,\ldots, N\},\\
    \sum_j \eta_k(j) &= 1. \notag 
  \end{align}

  Plugging \eqref{eq:29} into \eqref{eq:28}, it follows that there
  exists a set of functions $\mathcal H = \{H_1, \ldots, H_K\}$ (as
  many as possible values of $\mu_t$) and a sequence of \iid random
  variables $\{F_t \in \mathcal H: t=1,2,\ldots\}$ such that
  \begin{align}
    s_t &= F_t(s_{t-1}), \quad\quad\quad\mbox{for all }t, \label{eq:30}\\
    F_t &= H_k \quad\mbox{w.p. }d_k, \quad\mbox{for all }t,k. \notag 
  \end{align}
  where each $H_k$ corresponds to a realisation $\eta_k$ of the random
  variable $\mu_t$.  Hence, the process $\{s_t\}$ is generated by an
  iterated function system.

  \textbf{(Step 2)}\\
  In order to apply Proposition~\ref{thm:be}, 
  let us consider two signals $x, y \in \mathbb R^{2M}$, as defined in \eqref{defS}.
  We want to show that for 
  all $\eta_k$, we have
  $\norm{H_k (x) - H_k (y)}_1 \le \norm{x-y}_1$.  
  and for some $\eta_k$, we have
  $\norm{H_k (x) - H_k (y)}_1 < \norm{x-y}_1$.  
  The former is clear, whereas to show the latter, we need to establish that
  there exists $m$ such that, for all $t$, the event
    \begin{align*}
      \{ x_{2m-1} \geq c_m(n^m_t(x)) \mbox{ and } y_{2m-1} \geq c_m(n^m_t(y))) \}
    \end{align*}
    has positive probability.
    The above event corresponds to the event
    \begin{align*}
      \{ \underline u^m_{t+1}(x) = c_m(n^m_t(x))
      \mbox{ and } \underline u^m_{t+1}(y) = c_m(n^m_t(y)) \},
    \end{align*}
    which has positive probability for all finite $t$.

  We have, by definition \eqref{eq:28},
  \begin{align}
    \norm{ H_k (x) - H_k (y)}_1 \leq \sum_m & \abs{\min(x_{2m-1}, c_m(n^m_t(x)) ) - \min(y_{2m-1}, c_m(n^m_t(y)) )}\\
    &+ \abs{\max(x_{2m}, c_m(n^m_t(x)) ) - \max(y_{2m}, c_m(n^m_t(y)) )}, \notag
  \end{align}
  where $n^m_t(x)$ denotes the congestion profile at time $t$ when the
  signal $x$ is broadcast to all agents.  We denote the two summands
  on the right-hand side by $R_1, R_2$.  First, we bound $R_1$; the
  other summand $R_2$ can be bounded by a similar argument.  We have four
  cases:
  \begin{enumerate}
  \item $x_1 < c_m(n^m_t(x)) \mbox{ and } y_1 < c_m(n^m_t(y))$
  \item $x_1 \geq c_m(n^m_t(x)) \mbox{ and } y_1 \geq c_m(n^m_t(y))$
  \item $x_1 < c_m(n^m_t(x)) \mbox{ and } y_1 \geq c_m(n^m_t(y))$
  \item $x_1 \geq c_m(n^m_t(x)) \mbox{ and } y_1 < c_m(n^m_t(y)).$
  \end{enumerate}
  We only need to consider Case 2, which occurs with probability
  bounded away from zero by the above argument.

  Under Case 2, we have
  \begin{align*}
    R_1
    &= \abs{\min(x_1, c_m(n^m_t(x)) ) - \min(y_1, c_m(n^m_t(y)) )}\\
    &= \abs{c_m(n^m_t(x)) - c_m(n^m_t(y))},
  \end{align*}
  Observe that
  \begin{align}
    \abs{n^m_t(x) - n^m_t(y)} 
      = N \Big| \sum_{\omega \in \Omega} ( &1_{ \bigwedge_{s \neq m} (
        \omega x_{2m-1} + (1-\omega) x_{2m}  < \omega x_{2s-1}  + (1-\omega) x_{2s}  )} \notag \\
      &-
      1_{ \bigwedge_{s \neq m} ( \omega y_{2m-1} + (1-\omega) y_{2m}  < \omega
        y_{2s-1}  + (1-\omega) y_{2s} 
        )})\mu_t(\omega) \Big| \notag \\
    \leq N \sum_{\omega \in \Omega} \mu_t(\omega) \Big| & (1_{
        \bigwedge_{s \neq m} ( \omega x_{2m-1} + (1-\omega) x_{2m} < \omega
        x_{2s-1} + (1-\omega) x_{2s} )} \notag \\
      &- 1_{\bigwedge_{s \neq m} ( \omega y_{2m-1} + (1-\omega) y_{2m} <
        \omega y_{2s-1} + (1-\omega) y_{2s}
        )} ) \Big| \notag \\
    = N \sum_{\omega \in W} \mu_t(\omega),
  \end{align}
  where $W$ is an interval obtained by 
  solving the system of inequalities:
  \begin{align}
    \omega x_{2m-1} + (1-\omega) x_{2m} &< \omega
    x_{2s-1} + (1-\omega) x_{2s},\\
    \omega y_{2m-1} + (1-\omega) y_{2m} &<
        \omega y_{2s-1} + (1-\omega) y_{2s},\quad\mbox{for all }s \neq m. \notag 
  \end{align}
  Hence, there exists a constant $\kappa$ such that
  \begin{align*}
    \abs{W} \leq \kappa \norm{x-y}_1.
  \end{align*}
  In turn, we obtain
  \begin{align}
    \abs{n^m_t(x) - n^m_t(y)} &\leq N \sum_{ \omega \in W}  \mu_t(\omega) \\
    &\leq N \kappa \norm{x-y}_1 \max_\omega \mu_t(\omega). \notag 
  \end{align}

  Since $c_m$ is $\ell$-Lipschitz by assumption, it follows that
  \begin{align}
    \abs{c_m(n^m_t(x)) - c_m(n^m_t(y)) } &\leq \ell \abs{n^m_t(x) -
      n^m_t(y)} \notag \\
    &< \norm{x-y}_1,
  \end{align}
  where the last inequality follow from assumption.
  Finally, this allows us to verify that
  \begin{align}
    \sum_k d_k \log \left( \frac{\norm{H_k (x) - H_k (y)}_1}{\norm{x-y}_1}
    \right) < 0.\label{eq:cond}
  \end{align}
  Having verified \eqref{eq:cond}, the process $s_t$ converges in
  distribution to a unique limit by Proposition~\ref{thm:be}. In turn,
  the outcome $n^m_t$ converges likewise by \eqref{eq:29}. 
\end{proof}

\section{Simulations}
\label{sec:sim}

Although the case of infinite recall, $r(t) = t$, is amenable to analysis, the case
of finite recall is more realistic.  We hence simulate the finite recall case on 
a benchmark \cite{BarGera} for the traffic assignment problem, where the cost function captures
the travel time and each action corresponds to one path between two vertices. 
The travel time for path $P(i)$ of agent $i$ is a sum of travel times $d^e_t$ over edges $e \in P(i)$ at time $t$, where the travel time is the Bureau of Public Roads (BPR) function of the number $x^e_t$ of agents passing over $e$:
\begin{align}
d^e_t = F^e \cdot ( 1 + B^e \cdot (x^e_t/\chi^e)^{p^e}),
\label{eq:traveltime}
\end{align}
where   
$\chi^e$ is the capacity of $e$, $F^e$ is the free-flow time of $e$, $B^e$ and $p^e$ are constants, again particular to $e$,
often $B^e = 0.15, p^e = 4$.

For simplicity, we send out signals $(\underline u_t^e, \overline u_t^e)$ specific to each edge $e$,
rather than for each possible path.
We also replace the social cost by the agent- and edge-wise sum $\sum_{i = 1}^{N} \sum_{e \in P(i)} d^e_t$.
Although this set-up may seem rather arbitrary, a similar set-up has been used throughout hundreds of papers \cite{patriksson2008survey} on the 
traffic assignment problem in transportation science. 

On two instances, we show that the interval signaling we propose results in 
a regret, i.e. the distance to the social cost at the stochastic user equilibrium, which is convergent. On an artificial instance, which we call Diamond, the regret goes to 0 after a small
number of iterations. 
On the well-known Sioux Falls instance of LeBlanc et al. \cite{leblanc1975efficient}, we improve the social cost
considerably, when compared to the best known stochastic user equilibrium,
as reported by Hillel Bar-Gera \cite{BarGera}.


\subsection{The Procedural Details}

Let us now clarify a number of procedural details.
First, notice that without knowing the congestion profile $n_{t+1}$ at
time $t$, it is difficult to enforce the capacity constraint $x^e_t < \chi^e_t$.
Consequently, the term $(x^e_t/\chi^e)^{p^e}$ tends to 
produce outliers in terms of $u^e_t$ across all $e, t$, which are just modeling artifacts.
In most of our simulations, we hence apply a cap on the travel time:
\begin{align}
c^e_t = F^e \cdot ( 1 + B^e \cdot (\min\{x^e_t/\chi^e,1\})^{p^e}).
\label{eq:traveltimeCapped}
\end{align}
This eliminates the outliers, but makes it necessary to track the violation
of capacity constraints by other means.
To that end, we introduce the capacity excess:
\begin{align}
E^e = \begin{cases} x^e - \chi^e & \textrm{ if } x^e > 1 \\
                  0     & \textrm{ otherwise }
\end{cases}
\label{eq:excess}
\end{align}
which captures the aggregate amount of violation of the constraint $x^e_t < \chi^e_t$. 

We have disregarded tolls and distances, discretised time,
and proceeded as follows in each period:
\begin{enumerate}
\item Generate the population with size $N$ in $|\Omega|$ types, where we assume the $\Omega = (0, \frac{1}{|\Omega|-1}, \frac{2}{|\Omega|-1}, \ldots, 1)$
throughout. The proportion of each type in the population is sampled from the uniform distribution $U(\frac{1}{|\Omega|} - \epsilon, \frac{1}{|\Omega|} + \epsilon)$ for the
first $|\Omega|-1$ types, with the remainder for the final one.
Specifically, we use $|\Omega| = 5$  and $\epsilon =0.15$.
\item Generate signals $(\underline u_t^e, \overline u_t^e)$  
      for each $e$, depending on whether we cap the travel time,
      using the history of congestion cost up to $t$.
      If the history contains $n = 2$ or more per-link costs recorded, we use the minimum and maximum 
      within the $\min\{ r, n \}$ most recent travel times (possibly capped) for the edge.
      Otherwise, we use signal $(0, 0)$ for each path to initialise the simulation.
\item Compute the the number $x^e_t$ of agents passing over each edge $e$. For each $\omega \in \Omega$ and each origin-destination pair $(o,d)$ we pick acyclic paths 
\begin{align}
M = \arg \min_{p \in P((o,d))} \sum_{e \in p} \omega \underline u_t^e + (1 - \omega ) \overline u_t^e.
\label{eq:minpaths}
\end{align}
     where, with some abuse of notation, $P((o,d))$ are all acyclic paths between origin $o$ and 
     destination $d$.
      If there are multiple such paths, $|M| > 1$, we subdivide the number of agents of the given type that travel between the given $o$ and $d$ into $|M|$ equal parts $r$, which need not be a whole number.
      For each edge $e \in p$ on each path $p \in M$, we then add the part $r$ to the traffic to $x^e_t$.
\item Generate generalised per-link costs $c_{t}^e$ using \eqref{eq:traveltimeCapped} or $d_{t}^e$ using \eqref{eq:traveltime}.
      Add those to the history for future use. 
\item Generate per-path costs. For each origin-destination pair $(o,d)$,
      we again consider paths $M$ as in \eqref{eq:minpaths} and sum up the per-edge travel times. 
\item Compute the social cost, by summing up across all origin-destination pairs $(o,d)$ and 
      all paths $m \in M$ as in (\ref{eq:minpaths}),
      the product of the per-path cost, the proportion of the population corresponding to the path,
      and the cardinality of the population.
\item Move to the next period, $t = t +1$.
\end{enumerate}

This makes it possible to plot 
the evolution of the social cost $C$ 
over time
and the evolution of the sum of the excesses $E$ \eqref{eq:excess} across all links over time
for
scalar signaling using the most recent travel time (NOW), means of values seen so far (MEAN), 
and $r$-extreme signaling $r=5,10,20$.
In plots of the social cost and excess, we also plot the corresponding value of the
stochastic user equilibrium not considering information provision, either for 
the global optimum, where known, or for the best known equilibrium as reported by 
Hillel Bar-Gera \cite{BarGera}.

\subsection{The Diamond Instance}

First, we present experiments on an instance on five nodes, 1, 2, \ldots, 5,
with five links 1-2, 2-3, 2-4, 3-5, and 4-5, which form a ``diamond shape''.
There, links 3-5 and and 4-5 have high very high capacity
and identical cost functions.
Each of the links 2-3 and 2-4 
can carry half of the total traffic,
 but their cost functions differ markedly, as suggested 
columns 3 and 5--7 in Table~\ref{tab:art1nettxt}.
These two files presented in Table~\ref{tab:art1nettxt} can be provided 
as an input to a variety of tools developed in transportation engineering,
and hence allow for cross-comparison and cross-validation of our results.




\begin{table}
\begin{minipage}[t]{0.45\textwidth}
\begin{tabular}{rrrrrrrrrr}
\rothead{From} & \rothead{To} & \rothead{$\chi^e$} & \rothead{Length} & 
\rothead{$F^e$} & \rothead{$B^e$} & \rothead{$p^e$} \\
\hline \\[1mm]
1 & 2 & 25900 & 6 & 6 & 0.15 & 4 \\
2 & 3 & 15 & 0 & 2 & 1 & 2 \\
2 & 4 & 15 & 0 & 2 & 10 & 6  \\
3 & 5 & 99900 & 6 & 1 & 0.15 & 1  \\
4 & 5 & 99900 & 6 & 1 & 0.15 & 1  
\end{tabular}
\end{minipage}
\hskip 24mm
\begin{minipage}[t]{0.45\textwidth}
\begin{verbatim}
<NUMBER OF ZONES> 5
<TOTAL OD FLOW> 30
<END OF METADATA>
Origin 	1 
    5 :    30;
\end{verbatim}
\end{minipage}
\caption{Left: net.txt of the diamond instance, except for columns
Speed, Toll, and Type, whose values are uniformly 0, 0, 1.
Right: trips.txt of the diamond instance.
}
\label{tab:art1nettxt}
\end{table}


\begin{figure}
\centering
\includegraphics[width=0.49\textwidth]{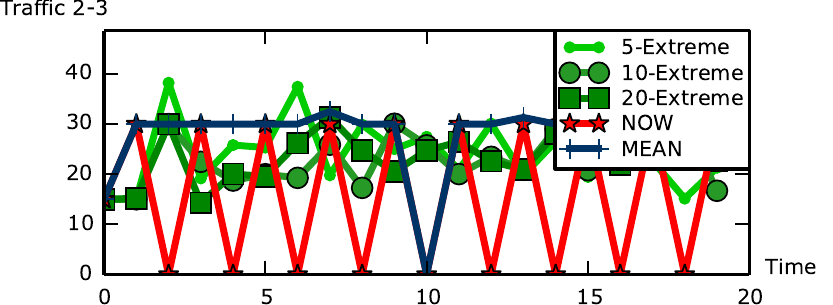}
\includegraphics[width=0.49\textwidth]{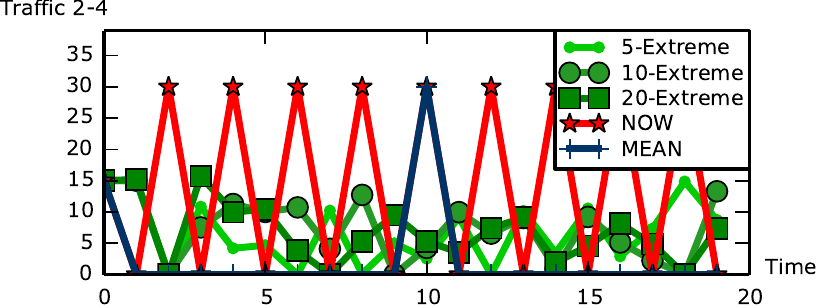}\\[2mm]
\includegraphics[width=0.49\textwidth]{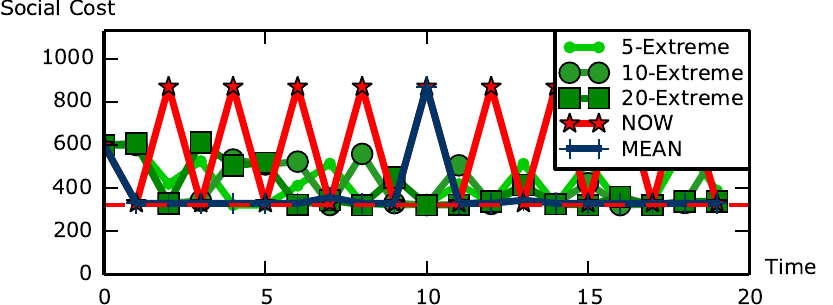} 
\includegraphics[width=0.49\textwidth]{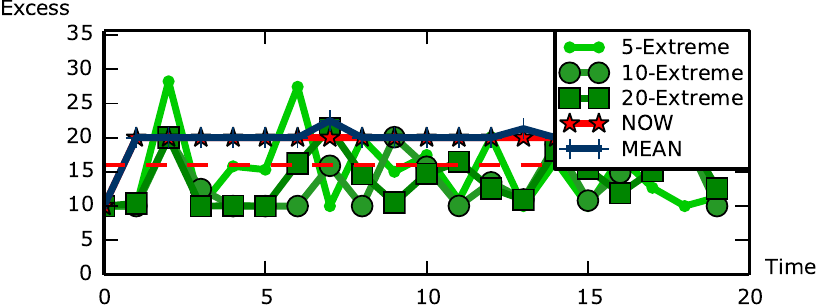}
\caption{
The evolution of traffic on links 2-3 and 2-4 and
the corresponding social cost and capacity excesses using capped travel times $c^e_t$ over time for 
scalar signaling using the most recent travel time (NOW) and means over all travel times (MEAN), 
compared with $r$-extreme signaling $r=5,10,20$ on the diamond instance.
The red dashed line corresponds to the best-known stochastic user equilibrium.}
\label{fig:art1traffic}
\end{figure}


The diamond instance illustrates the phenomenon of flapping well. 
The split of the traffic across links 2-3 and 2-4 is illustrated in the upper half of Figure~\ref{fig:art1traffic},
where for scalar signaling (NOW, MEAN), the traffic oscillate between paths 1-2-3-5 and 1-2-4-5,
whereas the higher the $r$, the smaller are the period-to-period changes for $r$-extreme signaling.
This corresponds to much lower social cost and  capacity excesses  
for $r$-extreme signaling,
compared to scalar signaling using means or most recent values,
as suggested in the bottom half of Figure~\ref{fig:art1traffic}.

Further, notice that the social cost approaches that of the best-possible 
stochastic user equilibrium, highlighted by the red dashed line in Figure~\ref{fig:art1traffic}.
The unique minimum of the un-capped cost at the stochastic user equilibrium, without considering information provision,  
of approx. 621.229 can be found 
by minimising 
$(2-x) (1 + (2-x)^2) + x (10 + x^6))$
over the interval $[0, 2]$.
The corresponding capped cost is 322.307 
and excess 15.985.
Hence, the regret approaches 0, in this particular case.


In Figure~\ref{fig:art1traffic}, we have capped the value of the travel time at the
value given by the travel time at capacity and counter the excess separately.
When we do not cap the travel time at capacity, the behavior in terms of the proportions of traffic
going either way is similar, as can be seen by comparing Figures~\ref{fig:art1traffic} and \ref{fig:art1trafficNotCapped},
while the absolute difference between the social costs of using the most recent time and $r$-extreme signaling
increases with the number of agents on the road.

\begin{figure}
\centering
\includegraphics[width=0.49\textwidth]{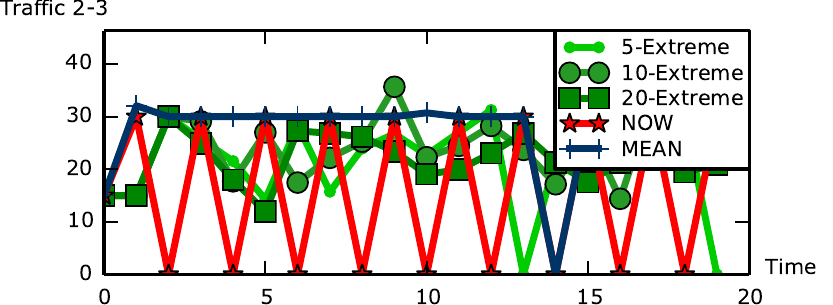}
\includegraphics[width=0.49\textwidth]{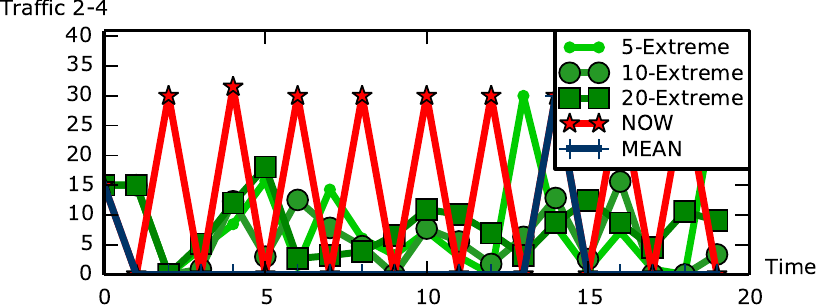}
\caption{
The evolution of traffic on links 2-3 and 2-4 over time on the diamond instance,
as in Figure~\ref{fig:art1traffic}, except using the uncapped travel time $d^e_t$.
}
\label{fig:art1trafficNotCapped}
\end{figure}

\subsection{Sioux Falls}

Next, we have tested the signaling on the well-known Sioux Falls instance
of LeBlanc et al. \cite{leblanc1975efficient}, displayed in Figure~\ref{fig:siouxNet}.
Since 1970s, this instance has attracted much attention in the transportation engineering
community \cite{Morlok1973,leblanc1975efficient,abdulaal1979,dantzig1979},
serving as a benchmark for the traffic assignment problem.
In particular, we have used the variant distributed by 
Bar-Gera \cite{BarGera},
which corresponds to 360,600 agents moving through a network of 76 road segments with 24 junctions.

The best-known stochastic user equilibrium,
 as available from Bar-Gera \cite{BarGera}
 has capped cost of 3853754.650
with excess of 265068.520
and un-capped cost of 7480225.345
with the same excess.
(Notice that these numbers vary from those reported by Bar-Gera, 
considering our objective functions differ.)
With cap on the travel time given by the capacity, as above \eqref{eq:traveltimeCapped}, 
the use of $r$-extreme signaling leads to lower social cost 
with lower excess, as suggested in Figure~\ref{fig:siouxSocial}.
Without cap on the travel time, the results are more varied,
and heavily skewed by a small number of enormous values.
Consequently, $r$-Extreme signaling seems to perform the best for $r=2$,
although this surprising behavior merits further study.

\begin{figure}
\centering
\includegraphics[width=0.49\textwidth]{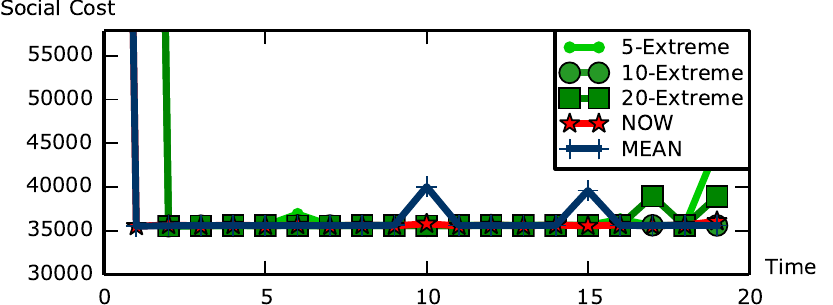} 
\includegraphics[width=0.49\textwidth]{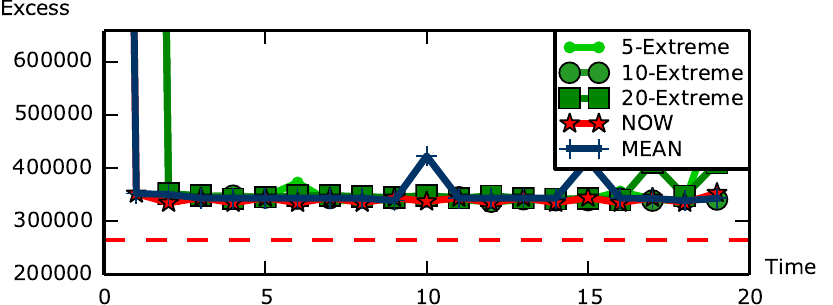}
\caption{
The social cost and capacity excess using capped travel times $c^e_t$
over time for 
scalar signaling using the most recent travel time (NOW) and means over all travel times (MEAN), 
compared with $r$-extreme signaling for $r=5,10,20$ on Sioux Falls.}
\label{fig:siouxSocial}
\end{figure}

\section{Related Work}
\label{sec:related}

There is related work being done in applied probability, control, operations
research, theoretical computer science, and traffic theory. 
There are number of excellent surveys 
\cite{sheffi1985urban,ibaraki1988resource,stefanov2001separable,patriksson2008survey,gittins2011multi}
available, although it may be difficult even for a book-length survey to be 
fully comprehensive.

Within applied probability,
the rich history of work on the 
multi-action restless bandits problem, e.g.,  \cite{whittle1988restless,weber1990index,bertsimas2000restless},
has been summarised by Gittins et al. \cite{gittins2011multi}.
See the work of Glazebrook \cite{glazebrook2011,glazebrook2015asymptotic}
for some of the present-best results.
The replacement of a single scalar of feedback per arm played has been 
  suggested \cite{MannorShamir11,alon2013bandits,alon2015online} in the bandits literature,
often in connection to revealing the outcome of further arms as well.
We are not aware of any bi-level extensions, e.g., seeing the problem from the point of view of the owner of the bandit.


Within game theory, the social cost is the metric of a number of studies 
\cite{Roughgarden2002,Perakis2004,Correa2005,Bhaskar2010},
which show that, even when agents have full information,
a natural equilibrium outcome can incur much higher total congestion than a
socially optimal outcome.
This is known as the price of anarchy.
Particularly interesting are studies
of the Nash equilibria in connection with ignorance
\cite{balcan2009price,alon2010bayesian,Snowball}, often
concerning the number of players
\cite{ashlagi2007,ashlagi2009two,AlonMT13}, failures of agents
\cite{meir2012congestion}, failures of resources
\cite{penn2007congestion,penn2011congestion}, or stratified
and risk-averse populations \cite{6736571,Piliouras2013}.  
Indeed, our work can be seen as showing the benefits of 
ignorance to a stratified and risk-averse population,
albeit over the long run.
We can hence describe the attractor, whose existence is often moot
in the studies of Nash equilibria.
See \cite{Cantarella1995} for further well-developed arguments why
considering the fixed-points of a dynamical system is preferable
to the study of the Nash equilibrium.

Within economics, our work is reminiscent of the equilibrium
outcome of Sobel \cite{banks1987equilibrium} in the context of signaling games. See  \cite{sobel2012signaling} for an up-to-date survey.  
Our work is also reminiscent of large bodies of work on 
follow-the-perturbed-Leader \cite{hannan1957approximation}, trembling-hand equilibrium \cite{Selten1975},
and stochastic fictitious play \cite{Harsanyi1973},
inasmuch we also study repeated decisions and that
the decisions are random variables. 
However, our scheme separates the decision making of the
central authority from the decision making of the agents, 
and uses non-trivial procedures for the former decisions.  

In the transportation literature, 
Daganzo and Sheffi \cite{daganzo1977stochastic} have introduced the
concept of stochastic user equilibrium, where users have considerable 
amounts of information, and perhaps surprising analytical powers.
Subsequently, a number of variants have been proposed, e.g., 
\cite{bertsimas2003robust} consider robust variants and
\cite{ahipasaoglu2014beyond} consider the stochastic user equilibrium with 
distributional uncertainty over the travel times.
\cite{knoop2008traffic} consider a stratified and risk-averse population,
but only as much as a link failure is concerned.
There are a number of other notions of stability, perhaps closer
to our notion, inasmuch they capture the repeated nature of the problem.
For example, \cite{smith1979} introduces the notion
of equilibrium as the limit of the congestion distribution if it
exists.  \cite{horowitz1984} considers a number of notions of noisy
signals and studies greedy policies and equilibria.  Our approach is
different from those that assign actions to agents,
instead of presenting them with information and letting them make the
decisions.

On the interface of transportation and control theory, 
there has been a recent interest in load balancing \cite{ScholteKing2014,ScholteChen2014}.
These schemes, however, rely on simple randomisation, without 
modelling heterogenous agent behaviour and actions
and without allowing for the same information to be provided to all agents.
On the interface of transportation and behavioral science, 
Ben-Akiva et al. \cite{ben1991dynamic} have studied the effects
of information on drivers.
In a number of subsequent papers \cite{kaufman1998user,ben2001route} and 
the dissertation of Bottom \cite{bottom2000consistent}, fixed points have been used to
study deterministic scalar signaling with deterministic 
response of the population. 
See \cite{papageorgiou2007its} for an extensive survey.
In contrast, our analysis can be seen as a study of a probabilistic counter-part of 
fixed points, which allows for the uncertainty in the response of the population.

Throughout, we are not aware of any theoretical guarantees on the behavior
of policies similar to ours, 
 as described in this paper and
 \cite{marevcek2015signaling}.
Specifically, we are not aware of any other paper, which would study the broadcasting of 
intervals, instead of scalars,
show its superiority,
or study the behaviour of systems, where such signals are being provided.
Compared to this paper, the set-up of \cite{marevcek2015signaling} is much simpler, 
and so are the proofs and simulations. 
Unlike \cite{marevcek2015signaling}, the approach presented in this paper does not employ randomisation, whose use may be unacceptable to the general public,
allows for the same signal to be broadcast to all agents,
such as at road-side displays,
and considers a more elaborate model of the populational response,
with risk-aware agents.
Both this paper and \cite{marevcek2015signaling} suggest the importance of the 
control-theoretic aspects of information provision.

\section{Conclusion and Future Work}
\label{sec:conclusions}

We have introduced a novel interval signaling scheme.
As opposed to scalar signaling schemes, interval signaling schemes
have tremendous potential in reducing the social cost of congestion
and present a major step forwards in a number of applications, which allow for
agent-based models. 
This includes transportation and congestion management more broadly.
These applications also open a number of questions throughout the possible applications
of the approach as well as within cognitive science and control theory.  

Key questions in cognitive science include: 
To what extent do human populations react to any signals?  
How do human populations react to interval signals?  
What are the factors to consider in modelling the populational response, 
outside of the risk-aversion?
What incentives would be most appropriate in improving the response?
Answers to such questions should be of considerable interest to the optimisation
and control communities.

Key questions in control theory include:
Can our stability result be extended to other non-scalar signals, e.g., histograms?
Can our stability result be extended to more general stochastic populations, e.g.,
 $\mu_t$ evolving as a Markovian process?
How to reason about policies, where the intervals are obtained by optimisation over the interval
signals to send in the following period, subject to the signals being truthful in some sense?
Perhaps most importantly: 
the model could be seen as a bi-level optimisation problem, with the information provision
at the upper level and the choice of action at the lower level. 
For bi-level optimisation problems,
even solving the first-order optimality conditions \cite{Jeyakumar2015,Nie2015} presents a major challenge, whereas our approach provides certain guarantees for a certain solution to a certain bi-level optimisation problem. Could this be generalised?
We hope to answer some of these questions in due course.

Finally, one could consider further applications.
What is the performance of interval signaling schemes beyond transport applications, 
  e.g., in ad keyword auctions, electricity consumption time slots, and emergency evacuation routes? 
Some could, indeed, be of considerable independent interest.


\FloatBarrier
\bibliographystyle{abbrv} %
\bibliography{traffic,experiments,bandits}





\clearpage
\appendix 

\section{An Analysis of Flapping}
\label{sec:app1}


The following proposition motivates the introduction of interval
signaling.  Specifically, it shows that interval signaling schemes
make it possible to all but get rid of a particularly bad cyclical
outcome, sometimes known as ``flapping'' in networking literature. 

\begin{proposition}[The Price of Flapping]
  \label{prop:flapping}
  For every number $J > 0$, $M=2$, and an odd integer $N \ge 3$, there
  exist functions $c_1, c_2$, a set $\Omega$, a population profile
  $\mu$, and an interval signaling scheme $\rho$ with social cost $C(n_{t+1}^\rho)$ at at $t+1$ such that for every
  scalar signaling scheme $\sigma$ with social cost $C(n_{t+1}^\sigma)$, we have $C(n_{t+1}^\rho) \leq
  C(n_{t+1}^\sigma) - J$.
\end{proposition}

The example used in the proof of Proposition~\ref{prop:flapping} may
seem extreme, but extensive simulations, which we have conducted, do
suggest that the cyclic behavior encountered in scalar signaling is
indeed reduced to a large extent, when one applies interval signaling.

\begin{proof}
  For an arbitrary constant $J$, let us construct cost functions $c_1, c_2$, 
  for two actions, where
  the difference in the social cost of the resulting congestion
  profiles
  \begin{align*}
    (n_t^1, n_t^2) \in & O_1 \triangleq \{ (0, N), (N, 0)\} \\ 
    (n_t^1, n_t^2) \in & O_2 \triangleq \{ \lfloor N/2 \rfloor, \lceil N/2 \rceil, \lceil N/2 \rceil, \lfloor N/2 \rfloor \}
  \end{align*}
  is $J$.  Consider
  \begin{align*}
    c_1(n) = c_2(n) = \begin{cases}
      \; 1 & \text{ for } n < \frac{N + 1}{2} \\
      \; (J+1)^{(2n-N)/N} & \text{ for } n \ge \frac{N + 1}{2}.
    \end{cases}
  \end{align*}
  The optimum of the social cost $C$ is clearly achieved for
  congestion profiles such that $\{n_t^1, n_t^2\} = 
  \{\lfloor N/2 \rfloor, \lceil N/2 \rceil\}$.  

  Let $\mu \triangleq \mu_t$ be deterministic for all $t$.  For
  interval signaling, observe that:
  \begin{align*}
    n_t^1 &= \sum_{\omega \in \Omega} 1_{[\pi^{\omega}(s_t) = A_m]} N \mu(\omega)\\
    &= \sum_{\omega \in \Omega} 1_{[\omega \underline{u}^1_t +
      (1-\omega) \overline{u}^1_t < \omega \underline{u}^2_t -
      (1-\omega) \overline{u}^2_t]} N \mu(\omega),
  \end{align*}
  which is possible to solve for $\Omega, \underline{u}^1_t,
  \overline{u}^1_t, \underline{u}^2_t, \overline{u}^2_t$ such that,
  \eg $n_t^1 = \lceil N/2 \rceil$, even considering that the interval
  signaling is $r$-subinterval \eqref{eq:supported4} and 
  $r$-extreme interval signaling $\rho$ with $r=2$.
  We can hence find a singleton $\Omega$ and an initial signal $s_1 \in \R^4$
  such that $(n_1^1, n_1^2) \in O_2$ by the argument above.  
  This means we do observe a cyclic behavior,
  but that is limited to elements of $O_2$, i.e. the best possible congestion profile,
  up to the rounding.

  In contrast, recall that the scalar signaling scheme is equivalent to
  interval signaling scheme with $\underline{u}^m_t =
  \overline{u}^m_t \forall m$, when the agents follow the
  policies $\pi^\omega$ for any $\omega$. For all $\omega, \omega'$,
  we have $\pi^\omega=\pi^{\omega'}$, $\omega$ thus becomes
  irrelevant, and hence we have:
  \begin{align*}
    n_t^1 &= \sum_{\omega \in \Omega} 1_{[\underline{u}^1_t <
      \underline{u}^2_t]} N \mu(\omega) \in \{ 0, N \}.
  \end{align*}
  For any scalar signaling scheme $\sigma$,
  the congestion profile $(n_t^1, n_t^2)$ will hence alternate
  between ``all-or-nothing'' elements of $O_1$.
  We call this cyclic behavior ``flapping''.

  Hence, $C(n_{t+1}^\sigma) = J+1, C(n_{t+1}^\rho) = 1$, and $C(n_{t+1}^\sigma)
  - C(n_{t+1}^\rho) = J$.
\end{proof}

%



\begin{figure}
\centering
\includegraphics[width=0.69\textwidth]{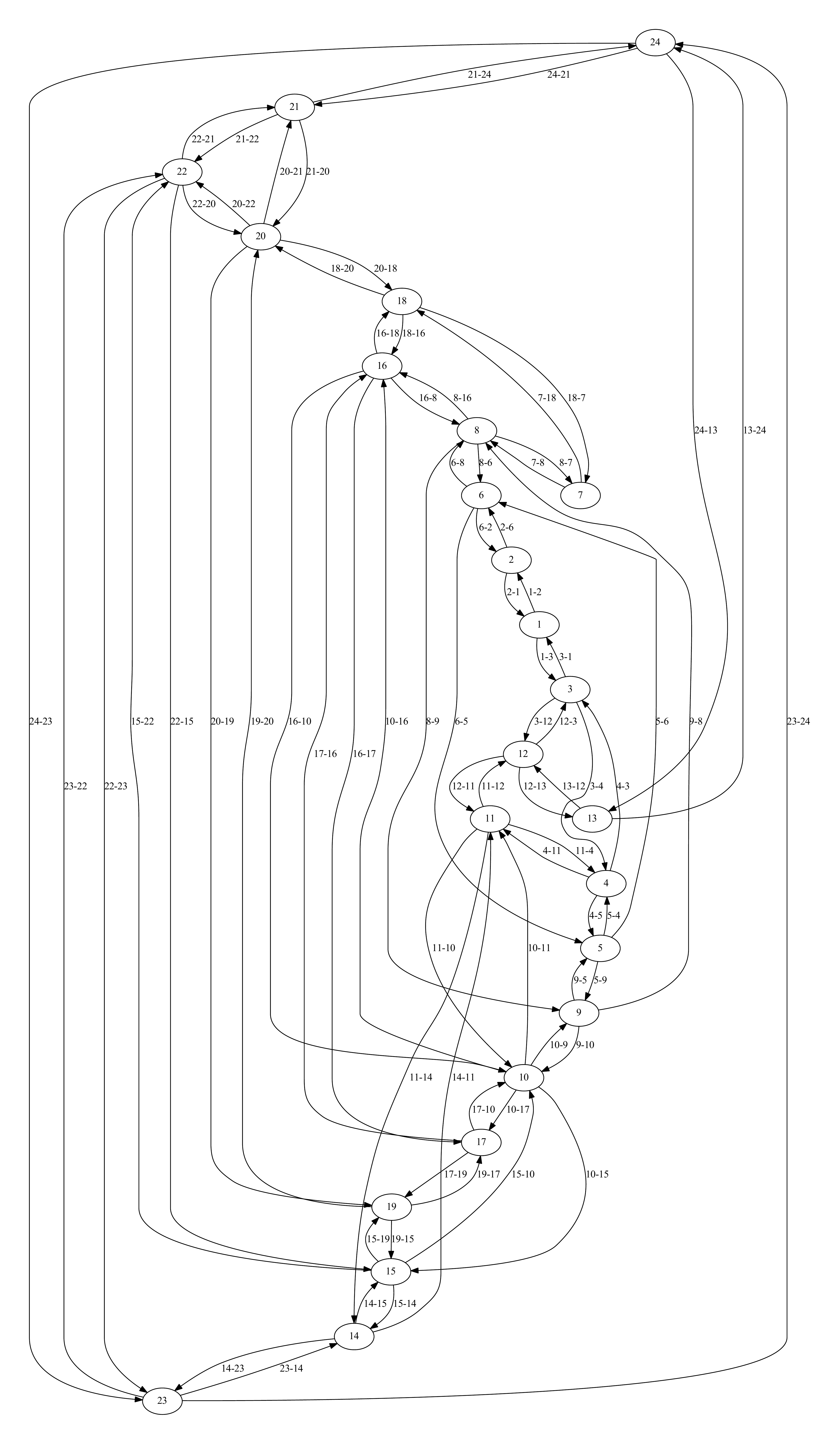}
\caption{A sketch of the Sioux Falls network.}
\label{fig:siouxNet}
\end{figure}

\end{document}